\documentclass[a4paper]{article}
\usepackage[utf8]{inputenc}

\usepackage{geometry}
\usepackage{lineno,hyperref}
\usepackage{comment}
\modulolinenumbers[5]
\usepackage[dvipsnames]{xcolor}
\definecolor{mathieu}{rgb}{0.5, 0.7, 0.8}

\newcommand{\scal}[2]{\left\langle {#1} \middle| {#2} \right\rangle}
\newcommand{\sigmaD}[1]{\underline{\sigma} \big(#1\big) } 
\newcommand{\sigmaU}[1]{\bar{\sigma} \big(#1\big) } 

\usepackage{amssymb,mathtools}
\usepackage{amsthm}
\newtheorem{thm}{Theorem}

\newtheorem{assum}{Assumption}
\newtheorem{lem}{Lemma}
\newtheorem{cor}{Corollary}
\newtheorem{rem}{Remark}

\usepackage{epsfig}
\usepackage{subcaption}

\begin{document}

\title{Design of low-dimensional controllers for high-dimensional systems}
\date{October 2023}

\author{Mathieu Bajodek, Hugo Lhachemi, Giorgio Valmorbida}


\maketitle


\begin{abstract}
This article presents proposals for the design of reduced-order controllers for high-dimensional dynamical systems. The objective is to develop efficient control strategies that ensure stability and robustness with reduced computational complexity. By leveraging the concept of partial pole placement, which involves placing a subset of the closed-loop system's poles, this study aims to strike a balance between reduced-order modeling and control effectiveness. The proposed approach not only addresses the challenges posed by high-dimensional systems but also provides a systematic framework for controller design. When an infinite-dimensional operator is Riesz spectral, our theoretical analysis highlights the potential of partial pole placement in advancing control design. Model uncertainties, introduced by an error on the spectral decomposition, can also be allowed. This is illustrated in particular in the case of systems modeled by coupled ordinary-partial differential equations (ODE-PDE).
\end{abstract}



\section{Introduction}
\label{sec0}

In modern engineering practice, the control of dynamical systems has become crucial due to their widespread applications across diverse domains such as aerospace, robotics, manufacturing, and power systems~\cite{Goodwin00}. The control design for dynamical systems faces significant growth in size and complexity over the years due to complex models and specifications. For instance, one can mention the interconnections of numerous subsystems, with subsystems described by partial differential equations. As the dimension of these dynamical systems grows, the design and implementation of controllers that guarantee stability, performance, and robustness have become hot challenges. Synthesizing a controller for a high-dimensional system requires the use of sophisticated methods~\cite{Silnaj05}.\\ 

Control design methods such as backstepping~\cite{Krstic08} and forwarding~\cite{Marx21} became the standard for linear infinite-dimensional systems. These approaches involve manipulating operators and constructing infinite-dimensional controllers~\cite{Bribiesca15,Meurer12}. However, when it comes to implementing such controllers, numerical issues arise due to the substantial computational resources required~\cite{Karafyllis17}. Moreover, the computational burden can lead to instabilities that are difficult to prevent~\cite{Mondie03}. 
An alternative approach involves reducing the order of the system and designing the controller using the reduced order model~\cite{Rivera87,Harkort14}. Although reduced-order controllers stabilize the reduced model effectively, ensuring the stability of the original high-dimensional system is challenging~\cite{Mironchenko20, Zhang10}.\\

In the context of dissipative infinite-dimensional systems, the main result focuses on the existence of an order from which the controller effectively stabilizes the closed-loop system, under the assumption that the reduced-order model accurately captures the rightmost system's spectrum~\cite{Morris20}. 
In the literature, different and more direct techniques are proposed for synthesizing fixed-order controllers with the goal of minimizing a specific cost while satisfying stabilization constraints~\cite{Yang07,Henrion03}. In an even more compact form, the design of proportional integral derivative controllers is performed for distributed parameter system~\cite{Alvarez12}, hyperbolic PDEs~\cite{Boussaada20}, parabolic PDEs~\cite{Lhachemi22} or ODE-PDE couplings~\cite{Boulouz21}. Implicitly, these approaches involve performing partial pole placement~\cite{Boussaada22}. For that matter, as addressed in~\cite{Datta12,Katewa21}, partial pole placement rewrites as optimization problems with constraints on the location of eigenvalues of the closed-loop system in predefined regions of the complex plane. By strategically placing a limited number of poles, reduced-order controllers focus on stabilizing and controlling the unstable dynamics of the system. A trade-off between computational complexity and control performance can then be reached.\\
This work delves into the design of reduced-order controllers using partial pole placement techniques for linear infinite-dimensional systems~\cite{Curtain82}. The main goal is to develop control strategies that not only reduce the computational burden but also ensure stability, performance, and robustness for the original high-dimensional dynamical system~\cite{Morris20,Jones10,Bergeling20}. The proposed approach provides a systematic framework for reduced-order observer-based control design.
It is based on a mild set of assumptions, outlined below:\vspace{-0.2cm}
\begin{itemize}
    \item the system is exactly controllable and observable, \vspace{-0.2cm}
    \item the system has bounded input and output operators, \vspace{-0.2cm}
    \item the system has a finite number of eigenvalues on the right of any vertical line in the complex plane,\vspace{-0.2cm}
    \item the ratio between the input-output operators upper bound and the state operator lower bound is bounded.
    \vspace{-0.2cm}
\end{itemize}
Taking inspiration from the approach introduced in~\cite{Balas83,Sakawa83}, an output feedback dynamic controller that reconstructs and stabilizes the unstable modes while preserving the stability of the remaining stable modes is constructed in two steps. 
On one hand, a static state feedback control gain is set up, which shifts the unstable eigenvalues with the same dimension as the unstable part of the initial system, which is denoted~$n_0$. On the other hand, a Luenberger-like observer of order $n\geq n_0$ is built~\cite{Luenberger64}. The estimation error on the unstable modes is rendered stable with observation gains, while the stable modes are simply reconstructed in open loop. It is then shown that there exists a sufficiently large order $n$ such that this finite-dimensional output feedback dynamic controller stabilizes the infinite-dimensional system. The stability is proven based on the construction of a Lyapunov function. The practical aspects of implementing these controllers are also investigated, taking into account uncertainties in the eigenvalues of input and output operators. It is proven that under small enough model mismatch, the result still holds.\\

Section~2 summarizes the problem and assumptions with two motivating examples. Section~3 proposes a controller structure of order $n$ and proves the stability of the closed-loop system for a sufficiently large order $n$.
Section~4 extends this result to the control synthesis subject to modeling errors. Section~5 finally apply numerically this common framework to specific examples such that ODE-transport and ODE-reaction-diffusion interconnections.\\

\textbf{Notation:} Throughout this article, $\mathbb{N}$, $\mathbb{R}$, $\mathbb{C}$, $\mathbb{R}^{n\times m}$, $\mathbb{C}^{n\times m}$ and $\ell_2(\mathbb{N})$ denote the sets of integers, real or complex numbers, real or complex matrices of size $n\times m$ and the space of square summable complex sequences, respectively. For any $s\in\mathbb{C}$, $\mathrm{Re}(s)$ and $\mathrm{Im}(s)$ stand for its real and imaginary parts, respectively. For any $M\in\mathbb{C}^{n\times m}$, $M^\ast$ is the conjugate transpose of $M$. For any square matrix $M$, $\sigmaD{M}$ and $\sigmaU{M}$ denote the minimal and maximal real parts of its eigenvalues, respectively. Furthermore, $M\succ 0$ denotes a positive definite symmetric matrix. Lastly, we define the following functions
\begin{equation*}
    \lvert \cdot\rvert:
    \left\{
    \begin{aligned}
        \mathbb{C}^n &\to \mathbb{R},\\
        x &\mapsto \sqrt{x^\ast x},
    \end{aligned}
    \right.\qquad
    \lVert \cdot\rVert:
    \left\{
    \begin{aligned}
        \ell_2(\mathbb{N}) &\to \mathbb{R},\\
        (x_i)_{i\in\mathbb{N}} &\mapsto \sqrt{\sum_{i\in\mathbb{N}} |x_i|^2},
    \end{aligned}
    \right.\qquad
    \scal{\cdot}{\cdot}:
    \left\{
    \begin{aligned}
        \ell_2(\mathbb{N})\times \ell_2(\mathbb{N}) &\to \mathbb{R},\\
        (x_i)_{i\in\mathbb{N}},(z_i)_{i\in\mathbb{N}} &\mapsto \sum_{i\in\mathbb{N}} x_i^\ast z_i.
    \end{aligned}
    \right.
\end{equation*}

\section{Preliminaries and assumptions}
\label{sec1}

\subsection{System data}

We focus on the following subclass of linear infinite-dimensional system
\begin{equation}\label{eq:systemref}
\left\{
\begin{aligned}
    \tfrac{\mathrm{d}}{\mathrm{d}t}x(t) &= \mathcal{A} x(t) + \mathcal{B} u(t),\\
    y(t) &= \mathcal{C} x(t),
\end{aligned}
\right.\quad \forall t\geq 0,
\end{equation}
where the operator $\mathcal{A}:D(\mathcal{A})\subset\ell_2(\mathbb{N})\!\to\! \ell_2(\mathbb{N})$ has generalized eigenvectors $\{\upsilon_k\}_{k\in\mathbb{N}}$ which form a Riesz basis~\cite{Gohberg78,Guo01} and where the operators $\mathcal{B}:\mathbb{R}^{n_u}\!\to\! \ell_2(\mathbb{N})$ and $\mathcal{C}:\ell_2(\mathbb{N})\!\to\! \mathbb{R}^{n_y}$ are bounded.

\begin{assum}\label{assum1}
For any $\delta>0$, there exists an order $n\in\mathbb{N}$ such that the generalized eigenvectors $\{\upsilon_k\}_{k\in\{n,\dots,\infty\}}$ are associated to algebraically simple eigenvalues with real parts smaller than $-\delta$.
\end{assum}

Under Assumption~\ref{assum1}, consider $(n_0,n_1)\in\mathbb{N}^2$ the number of eigenvalues of the operator~$\mathcal{A}$, repeated according to their multiplicity, which respectively
\begin{itemize}
    \item have real parts larger than $-\delta$, regarless of their algebraic multiplicity, 
    \item have real parts strictly smaller than $-\delta$ and are algebraically multiple.
\end{itemize}
For any integer $n\geq n_0+n_1$ (that will be specified later when discussing control design), system~\eqref{eq:systemref} takes, up to a change of variable, the following decomposed form
\begin{equation}\label{eq:system}
\left\{
\begin{aligned}
    \tfrac{\mathrm{d}}{\mathrm{d}t}\begin{bsmallmatrix}x_0(t)\\x_1(t)\end{bsmallmatrix} &= \begin{bsmallmatrix} A_0 & 0\\ 0 & A_1 \end{bsmallmatrix} \begin{bsmallmatrix}x_0(t)\\x_1(t)\end{bsmallmatrix} + \begin{bsmallmatrix}B_0\\B_1\end{bsmallmatrix} u(t),\\
    \frac{\mathrm{d}}{\mathrm{d}t}z_i(t) &= a_i z_i(t) + b_i u(t),\quad \forall i\in\mathbb{N},\\
    y(t) &= \begin{bsmallmatrix} C_0 & C_1 \end{bsmallmatrix} \begin{bsmallmatrix}x_0(t)\\x_1(t)\end{bsmallmatrix} + \sum_{i\in\mathbb{N}} c_i z_i(t),
\end{aligned}
\right.\quad \forall t\geq 0 .
\end{equation}
Here $(A_0,A_1,a_i)\in\mathbb{C}^{n_0\times n_0}\times \mathbb{C}^{n-n_0\times n-n_0}\times \mathbb{C}$ are on the complex Jordan normal form verifying $\sigmaU{A_1}< -\delta$ and $\mathrm{Re}(a_i)< \sigmaD{A_1} $, $(B_0,B_1,b_i)\in\mathbb{C}^{n_0\times n_u}\times \mathbb{C}^{n-n_0\times n_u}\times\mathbb{C}^{1\times n_u}$, $(C_0,C_1,c_i)\in \mathbb{C}^{n_y\times n_0}\times \mathbb{C}^{n_y\times n-n_0}\times\mathbb{C}^{n_y\times 1}$. The finite-dimensional part of the state $\begin{bsmallmatrix}x_0\\x_1\end{bsmallmatrix}$ belongs to $\mathbb{C}^n$ and the infinite-dimensional remainder $(z_i)_{i\in\mathbb{N}}$ belongs to $\ell_2(\mathbb{N})$.

\begin{rem}\label{rem3}
    Note that this decomposition is often used for the eigenvalue-based control of linear infinite-dimensional systems involving transport~\cite{Michiels03}, parabolic~\cite{Curtain82,Katz20}, wave~\cite{Lhachemi22} or beam~\cite{Han11} equations.
\end{rem}

\begin{rem}\label{rem4}
    Note that the framework of Assumption~1 is large enough to capture certain PDE-ODE interconnection, as discussed later in the paper. However, infinite-dimensional systems with accumulation points or rightmost asymptotic branch in the root locus do not satisfy Assumption~1. It is the case, for instance, of neutral time-delay systems with non Shur neutral parts~\cite{Gu03}.
\end{rem}

In the context of Assumption~1, the objective of this article is the design of a finite-dimensional controller (of dimension $n$) that takes the form: 
\begin{equation}\label{eq:cont}
\left\{
\begin{aligned}
    \tfrac{\mathrm{d}}{\mathrm{d}t}\hat{x}(t) &= L \hat{x}(t) + M y(t) + N u(t),\\
    u(t) &= K \hat{x}(t)
\end{aligned}
\right.
\end{equation}
where $L\in\mathbb{R}^{n\times n}$, $M\in\mathbb{R}^{n\times n_y}$, $N\in\mathbb{R}^{n\times n_u}$ and $K\in\mathbb{R}^{n_u\times n}$ are matrices to be designed such that~\eqref{eq:cont} stabilizes~\eqref{eq:system} with the exponential decay rate $\delta$, meaning that
\begin{equation}\label{eq:defstab}
    \exists \kappa\geq 1,\; \delta>0\,;\quad |x(t)|^2 + \lVert z_i(t) \rVert^2 + |\hat{x}(t)|^2 \leq \kappa \exp(-\delta t)\left(|x(0)|^2 + \lVert z_i(0) \rVert^2 + |\hat{x}(0)|^2 \right) ,\quad\forall t \geq 0. 
\end{equation}

\begin{rem}\label{rem5}
    Owing to the consideration of Riesz spectral systems and to Assumption~\ref{assum1}, there exists a continuous and bijective linear transformation~$\mathcal{T}:\ell_2(\mathbb{N})\to \mathbb{C}^n\times\ell_2(\mathbb{N})$ from system~\eqref{eq:systemref} to system~\eqref{eq:system} linked to the Jordan normalization and the rearranging by decreasing real parts. It satisfies
    $$\mathcal{T}x=\begin{pmatrix}\begin{bsmallmatrix}x_0\\x_1\end{bsmallmatrix}\\(z_i)\end{pmatrix},\quad \mathcal{T}\mathcal{A}\mathcal{T}^{-1}=\begin{pmatrix}\begin{bsmallmatrix}A_0&0\\0&A_1\end{bsmallmatrix}\\(a_i)\end{pmatrix},\quad \mathcal{T}\mathcal{B}=\begin{pmatrix}\begin{bsmallmatrix}B_0\\B_1\end{bsmallmatrix}\\(b_i)\end{pmatrix},\quad \mathcal{C}\mathcal{T}^{-1}=\begin{pmatrix}\begin{bsmallmatrix}C_0&C_1\end{bsmallmatrix}&(c_i)\end{pmatrix}.$$
    It is worth mentioning that the transformation $\mathcal{T}^{-1}$ is also continuous according to the open mapping theorem~\cite{Conway10}.
    Hence, the exponential stability of systems~\eqref{eq:systemref}-\eqref{eq:cont} and~\eqref{eq:system}-\eqref{eq:cont} is equivalent. Throughout this paper, for simplicity reasons, we will only focus on the system~\eqref{eq:system}-\eqref{eq:cont} and on the exponential stability defined in~\eqref{eq:defstab}. 
\end{rem}

\subsection{Controllability and observability}

The possibility to successfully design an output feedback controller \eqref{eq:cont} for \eqref{eq:system} to achieve exponential stabilization with decay rate $\kappa > 0$ in the sense of \eqref{eq:defstab} heavily relies on the controllability and observability properties of the finite-dimensional part of dimension $n_0$ of the plant. We discuss here the controllability and observability properties of the pairs $(A_0,B_0)$ and $(C_0,A_0)$, respectively.


\begin{assum}\label{assum2}
    The pair $(\mathcal{A},\mathcal{B})$ is exactly controllable and that the pair $(\mathcal{C},\mathcal{A})$ is exactly observable. 
\end{assum}

\begin{lem}
    Under Assumption~\ref{assum2}, the pair $(A_0,B_0)$ is controllable and the pair $(C_0,A_0)$ is observable.
\end{lem}
\begin{proof} 
    Only the observability proof is provided, since controllability is obtained by duality using the same arguments.\newline
    Let $\lambda>0$ larger than the supremum of the eigenvalues real parts of the operator $\mathcal{A}$. Applying~\cite[Thm~6.5.3]{Tucsnak09} with the exponentially stable shifted semigroup associated to $\mathcal{A}-\lambda I$, if $(\mathcal{C},\mathcal{A})$ is exactly observable, then there exists $k>0$ such that for all $s\in\{\mathbb{C}\,|\,\mathrm{Re}(s)<\lambda\}$,
    \begin{equation}\label{eq:ineq2}
    \frac{1}{|\mathrm{Re}(s)-\lambda|^2} \lVert (sI-\mathcal{A}) x\rVert^2 + \frac{1}{|\mathrm{Re}(s)-\lambda|} \lvert \mathcal{C}x \rvert^2 \geq k \lVert x\rVert^2, 
    \end{equation}
    The proof is then conducted by contradiction. Assume that $(C_0,A_0)$ is not observable. Applying the Hautus lemma, there exists $s\in\mathbb{C}$ and a non-null vector $x_0$ such that $(sI-A_0)x_0=0$ and $C_0x_0=0$. We have $\mathrm{Re}(s)\leq \sigmaU{A_0}<\lambda$. Using now the linear transformation $\mathcal{T}$ from~\eqref{eq:systemref} to~\eqref{eq:system} in Remark~\ref{rem5} and considering $x=\mathcal{T}^{-1} \begin{pmatrix}\begin{bsmallmatrix}x_0\\0\end{bsmallmatrix}\\(0)\end{pmatrix}$, we obtain $$\frac{1}{|\mathrm{Re}(s)|^2} \lVert (sI-\mathcal{A}) x\rVert^2 + \frac{1}{|\mathrm{Re}(s)-\lambda|} \lvert \mathcal{C}x \rvert^2  = \frac{1}{|\mathrm{Re}(s)-\lambda|^2} \left\lVert \mathcal{T}^{-1}\begin{pmatrix}\begin{bsmallmatrix}(sI-A_0)x_0\\0\end{bsmallmatrix}\\(0)\end{pmatrix} \right\rVert^2 + \frac{1}{|\mathrm{Re}(s)-\lambda|}\lvert C_0x_0\rvert^2 = 0.$$
    The inequality~\eqref{eq:ineq2} cannot be satisfied. Thus, the pair $(\mathcal{C},\mathcal{A})$ cannot be observable and the proof is concluded.
\end{proof}

\subsection{Motivating examples}

\subsubsection{ODE-transport interconnection}

Consider an ODE interconnected with the transport equation 
\begin{equation}\label{eq:system1}
\left\{
\begin{aligned}
    \tfrac{\mathrm{d}}{\mathrm{d}t}x(t) &= A x(t) + B z(t,0) + B_u u(t) ,\\
    \partial_{t}z(t,\theta) &= \tfrac{1}{h}\partial_\theta z(t,\theta),\;\forall \theta\in(0,1),\\
    z(t,1) &=  Cx(t),\\
    y(t) &= C_y x(t).
\end{aligned}
\right.
\end{equation}
where matrices $A\in\mathbb{R}^{n_x\times n_x}$, $B\in\mathbb{R}^{n_x\times 1}$, $B_u\in\mathbb{R}^{n_x\times n_u}$, $C\in\mathbb{R}^{1\times n_x}$, $C_y\in\mathbb{R}^{n_y\times n_x}$ and $h>0$.


\begin{lem}
    The generalized characteristic functions of system~\eqref{eq:system1} form a Riesz basis for $\mathcal{H}:=\mathbb{R}^{n_x}\times L_2(0,1)$.
    Moreover, system~\eqref{eq:system1} satisfies Assumption~\ref{assum1}.
\end{lem}

\begin{proof} See~\cite[Lemma~2.4.6]{Curtain95}. Indeed,
    solving the characteristic equation of system~\eqref{eq:system1}, the characteristic functions $(v_k)_{k\in\mathbb{N}}$ associated to $\{s_k\}_{k\in\mathbb{N}}$ are given by
    \begin{equation}\label{eq:xik1}
        v_k(\theta) = \begin{bmatrix} \mathrm{adj}(s_kI-A)B\\ C\mathrm{adj}(s_kI-A)B \exp\left(hs_k(\theta-1)\right)\end{bmatrix}\in\mathbb{C}^{n_x+1}.
    \end{equation}
    According to~\cite[Theorem~2.5.10]{Curtain95}, this sequence of functions $(v_k)_{k\in\mathbb{N}}$ is a Riesz basis in $\mathcal{H}$.     Assumption~\ref{assum1} is fulfilled as shown in~\cite[Lemma~4.2]{Hale93}.
\end{proof}

\subsubsection{ODE-reaction-diffusion interconnection}

Consider an ODE interconnected with the reaction-diffusion equation via Neumann inputs and Dirichlet output
\begin{equation}\label{eq:system2}
\left\{
\begin{aligned}
    \tfrac{\mathrm{d}}{\mathrm{d}t}x(t) &= A x(t) + B\partial_\theta z(t,1) + B_u u(t),\\
    \partial_{t}z(t,\theta) &= (\nu\partial_{\theta\theta}+ \lambda) z(t,\theta),\qquad\forall \theta\in(0,1),\\
    \begin{bsmallmatrix} z(t,0)\\z(t,1)\end{bsmallmatrix} &=  \begin{bsmallmatrix}Cx(t)\\0\end{bsmallmatrix},\\
    y(t) &= C_y x(t).
\end{aligned}
\right.
\end{equation}
where matrices $A\in\mathbb{R}^{n_x\times n_x}$, $B\in\mathbb{R}^{n_x\times 1}$, $B_u\in\mathbb{R}^{n_x\times n_u}$, $C\in\mathbb{R}^{1\times n_x}$, $C_y\in\mathbb{R}^{n_y\times n_x}$, $\nu>0$ and $\lambda\in\mathbb{R}$.

\begin{lem}
    The generalized characteristic functions of system~\eqref{eq:system2} form a Riesz basis for $\mathcal{H}:=\mathbb{R}^{n_x}\times L_2(0,1)$.
    Moreover, system~\eqref{eq:system2} satisfies Assumption~\ref{assum1}.
\end{lem}
\begin{proof}
See~\cite[Lemma~2]{Bajodek23}. Indeed, solving the characteristic equation of system~\eqref{eq:system1} as in~\cite[Appendix~A]{Bajodek23}, the characteristic functions $(v_k)_{k\in\mathbb{N}}$ associated to $\{s_k\}_{k\in\mathbb{N}}$ are given by
\begin{equation}\label{eq:xik2}
    v_k(\theta) = \begin{bmatrix} \mathrm{adj}(s_kI-A)B \sinh\left(\left(\frac{s_k-\lambda}{\nu}\right)^{\frac{1}{2}}\right)\\ C\mathrm{adj}(s_kI-A)B \sinh\left(\left(\frac{s_k-\lambda}{\nu}\right)^{\frac{1}{2}}(1-\theta)\right)\end{bmatrix} \in\mathbb{C}^{n_x+1}.
\end{equation}
Applying the modified Bari's theorem~\cite[Theorem~6.3]{Guo01}, we show that the sequence of functions $\{v_k\}_{k\in\mathbb{N}}$ forms a Riesz basis for $\mathcal{H}$. Assumption~\ref{assum1} is fulfilled as shown in~\cite[Proposition~1]{Zhao19}.
\end{proof}


\begin{rem}
    Note that Assumption~\ref{assum2} is not discussed for these two motivating examples. Indeed, to prove exact controllability and observability of linear infinite-dimensional systems is a tough task. Few results exist and concern only approximate controllability and observability, as in~\cite[Theorem~4.2.10]{Curtain95} for ODE-transport interconnection. In practice, we will directly check the controllability and observability of the pairs of matrices~$(A_0,B_0)$ and~$(C_0,A_0)$.
\end{rem}

\subsubsection{On the construction of the system data}

For the two above ODE-PDE interconnected systems, a projection into the Riesz basis generated by the generalized characteristic functions $\{v_k\}_{k\in\mathbb{N}}$ allows us to write the system dynamics under the form of \eqref{eq:system}. Indeed, defining $(\xi_k)_{k\in\mathbb{N}}$ as follows
\begin{equation*}
    \xi_k(t) = \int_{0}^1 v_k^\ast(\theta) \begin{bmatrix} x(t)\\z(t,\theta)\end{bmatrix} \mathrm{d}\theta \in\mathbb{C},
\end{equation*}
with $\{v_k^\ast\}_{k\in\mathbb{N}}$ satisfying $\int_0^1 v_k^\ast(\theta) v_i(\theta)\mathrm{d}\theta = \left\{\begin{array}{ll}1& \text{if } k=i,\\0&\text{otherwise,}\end{array}\right.$ we obtain
\begin{equation*}
\left\{
\begin{aligned}
    \dfrac{\mathrm{d}}{\mathrm{d}t}\xi_k(t) &= s_k \xi_k(t) + \int_{0}^1 v_k^\ast(\theta) \begin{bmatrix} B_u\\0\end{bmatrix} \mathrm{d}\theta u(t) ,\\
    y(t) &= \begin{bmatrix} C & 0 \end{bmatrix} \sum_{k\in\mathbb{N}} \xi_k(t) v_k(\theta)
\end{aligned}
\right.
\end{equation*}

Ordering the eigenvalues by decreasing real parts and bearing in mind that in practice the eigenvalues are algebraically simple in most cases, we can compute for a given order $n\in\mathbb{N}$ the following model
\begin{equation}\label{eq:ABC}
\begin{aligned}
\begin{bmatrix} A_0&0\\0&A_1\end{bmatrix}
&=\begin{bsmallmatrix}s_0&0&0\\0&\ddots&0\\0&0&s_n\end{bsmallmatrix}\in\mathbb{C}^{n\times n},\qquad
\begin{bmatrix} B_0\\B_1\end{bmatrix}=\begin{bmatrix}\int_{0}^1 v_1^\ast(\theta) \mathrm{d}\theta\\ \vdots\\\int_{0}^1 v_n^\ast(\theta) \mathrm{d}\theta\end{bmatrix} \begin{bmatrix} B_u\\0\end{bmatrix}\in\mathbb{C}^{n\times n_u},\\
\begin{bmatrix} C_0&C_1 \end{bmatrix} &= \begin{bmatrix} C & 0 \end{bmatrix} \begin{bmatrix}  v_1(\theta) & \cdots & v_n(\theta) \end{bmatrix}\in\mathbb{C}^{n_y\times n}.
\end{aligned}
\end{equation}
and $(a_i),(b_i),(c_i)$ are the remaining terms corresponding to the other eigenvalues.

\begin{rem}Similarly to Remark~\ref{rem5}, the bijective linear transformation from $\begin{bsmallmatrix}x\\z(\theta)\end{bsmallmatrix}$ towards $(\xi_k)$ being continuous, the exponential stability of the two above ODE-PDE interconected systems in terms of $\mathcal{H}:\mathbb{R}^{n_x}\times L_2(0,1)$ norm is equivalent to the one which is regarded in terms of $\ell_2(\mathbb{N})$~norm.
\end{rem}

\begin{rem}
In simulation, the solution is displayed using the relation $$\begin{bmatrix} x(t)\\z(t,\theta)\end{bmatrix} =\sum_{k\in\mathbb{N}} \xi_k(t) v_k(\theta).$$
For computational issues, these infinite-dimensional models are truncated at a sufficiently large order $N\gg n$, which explains the name of high-dimensional systems. 
\end{rem}




\section{Synthesis of the controller}

In this section, we assume that the system operators $(A_0,A_1,B_0,B_1,C_0,C_1)$ are perfectly known.

\subsection{Controller data}

The adopted control strategy consists of designing a state-feedback augmented with a Luenberger observer on the finite dimensional part of the system of dimension $n$ that describes the state variables $x_0$ and $x_1$ from \eqref{eq:system}. More precisely, we consider
\begin{equation}\label{eq:cont1}
\left\{
\begin{aligned}
    \frac{\mathrm{d}}{\mathrm{d}t}\begin{bsmallmatrix}\hat{x}_0(t)\\\hat{x}_1(t)\end{bsmallmatrix} &= \begin{bsmallmatrix} A_0 & 0\\ 0 & A_1 \end{bsmallmatrix} \begin{bsmallmatrix}\hat{x}_0(t)\\\hat{x}_1(t)\end{bsmallmatrix} + \begin{bsmallmatrix}B_0\\B_1\end{bsmallmatrix} u(t) + \begin{bsmallmatrix}G_0\\0\end{bsmallmatrix} (\hat{y}(t)-y(t)) \\
    \hat{y}(t) &= \begin{bsmallmatrix} C_0 & C_1 \end{bsmallmatrix} \begin{bsmallmatrix}\hat{x}_0(t)\\\hat{x}_1(t)\end{bsmallmatrix} , \\
    u(t) &= K_0 {x}_0(t),
\end{aligned}
\right.
\end{equation}
for some suitable control gains $K_0\in\mathbb{R}^{n_u\times n_0}$ and observation gains $G_0\in\mathbb{R}^{n_0\times n_y}$ to be fixed later. Hence, the above finite-dimensional controller can be written as~\eqref{eq:cont} with matrices 
\begin{equation}\label{eq:matrix1}
\begin{aligned}
    L &= \begin{bmatrix}
        A_0+G_0C_0 & G_0C_1\\
        0 & A_1
    \end{bmatrix},\quad
    M = \begin{bmatrix}
        -G_0\\ 0
    \end{bmatrix},\quad
    N = \begin{bmatrix}
        B_0\\ B_1
    \end{bmatrix},\quad
    K = \begin{bmatrix}
        K_0 & 0
    \end{bmatrix}.
\end{aligned}
\end{equation}

For later analyses, we divide the above controller structure~\eqref{eq:cont1} into two parts while including the errors of estimation of the observer.\\
The first part is composed of the reconstructed state $\hat{x}_0\in\mathbb{R}^{n_0}$ and the error $e_0=x_0-\hat{x}_0\in\mathbb{R}^{n_0}$. It follows the dynamics
\begin{equation}\label{eq:dyn1}
    \frac{\mathrm{d}}{\mathrm{d}t} \begin{bmatrix} \hat{x}_0 \\ e_0 \end{bmatrix} = F_0 \begin{bmatrix} \hat{x}_0 \\ e_0 \end{bmatrix} + \begin{bmatrix} -G_0  \\ G_0 \end{bmatrix} C_1e_1 + \begin{bmatrix} -G_0 \\ G_0\end{bmatrix} \sum_{i\in\mathbb{N}}c_iz_i,
\end{equation}
with $F_0=\begin{bmatrix} A_0+B_0K_0 & -G_0C_0 \\ 0 & A_0 + G_0C_0 \end{bmatrix}$.\\
The second part is composed of the reconstructed state $\hat{x}_1\in\mathbb{R}^{n-n_0}$ and the error $e_1=x_1-\hat{x}_1\in\mathbb{R}^{n-n_0}$, it follows the dynamics
\begin{equation}\label{eq:dyn2}
    \frac{\mathrm{d}}{\mathrm{d}t} \begin{bmatrix} \hat{x}_1 \\ e_1 \end{bmatrix} = \begin{bmatrix} A_1 & 0 \\ 0 & A_1 \end{bmatrix} \begin{bmatrix} \hat{x}_1 \\ e_1 \end{bmatrix} + \begin{bmatrix} B_1K_0\\0\end{bmatrix} \hat{x}_0.
\end{equation}\newline

Under Assumption~\ref{assum2}, the pair $(A_0,B_0)$ is stabilizable and the pair $(C_0,A_0)$ is detectable. Therefore, for any $\delta>0$, the gains $(K_0,G_0)$ can be selected such that there exists a symmetric positive matrix $P_0\in\mathbb{R}^{n_0\times n_0}$ which satisfy the Lyapunov inequality
\begin{equation}\label{eq:stab0}
    P_0 F_0 + F_0^\ast P_0 \prec -2\delta P_0.
\end{equation}

Under Assumption~\ref{assum1}, the remaining system is inherently exponentially stable with a decay rate $\delta$. Therefore, for any fixed integer $n$, there exists a symmetric positive definite matrix $P_1\in\mathbb{R}^{n-n_0\times n-n_0}$ such that the following Lyapunov inequality holds
\begin{equation}\label{eq:stab1}
    P_1A_1 + A_1^\ast P_1 \prec -2\delta P_1.
\end{equation}

\subsection{Lyapunov functional}

For some scalars $\alpha,\beta,\gamma>0$, consider the following Lyapunov functional
\begin{equation}\label{eq:lyap}
    \mathcal{V}(\hat{x}_0,e_0,\hat{x}_1,e_1,z_i) = \underbrace{\alpha \begin{bmatrix} \hat{x}_0 \\ e_0 \end{bmatrix}^\ast \!\!P_0\begin{bmatrix} \hat{x}_0 \\ e_0 \end{bmatrix}}_{=\mathcal{V}_0(\hat{x}_0,e_0)} + \underbrace{\begin{bmatrix} \hat{x}_1 \\ e_1 \end{bmatrix}^\ast \begin{bmatrix}
        \beta P_1 & 0 \\ 0 & \gamma P_1
    \end{bmatrix}\begin{bmatrix} \hat{x}_1 \\ e_1 \end{bmatrix}}_{=\mathcal{V}_1(\hat{x}_1,e_1)} + \underbrace{\lVert (z_i)\rVert^2}_{=\mathcal{V}_2(z_i)},
\end{equation}
where the matrices $P_0,P_1$ are fixed in accordance with~\eqref{eq:stab0}-\eqref{eq:stab1}.\\

Computing  the time derivative of this functional along the trajectories of the closed-loop system~\eqref{eq:system}-\eqref{eq:cont} yields
\begin{equation*}
\begin{aligned}
    \frac{\mathrm{d}}{\mathrm{d}t}\mathcal{V}_0(\hat{x}_0,e_0) &= 2\alpha  \begin{bmatrix} \hat{x}_0 \\ e_0 \end{bmatrix}^\ast \!\! P_0 \left( F_0 \begin{bmatrix} \hat{x}_0 \\ e_0 \end{bmatrix} + \begin{bmatrix} -G_0 \\ G_0 \end{bmatrix} C_1e_1 + \begin{bmatrix} -G_0 \\ G_0\end{bmatrix} \sum_{i\in\mathbb{N}}c_iz_i\right),\\
    \frac{\mathrm{d}}{\mathrm{d}t}\mathcal{V}_1(\hat{x}_1,e_1) &= 2\beta \hat{x}_1^\ast P_1 \left( A_1 \hat{x}_1 + B_1K_0\hat{x}_0 \right) + 2\gamma e_1^\ast P_1A_1 e_1,\\
    \frac{\mathrm{d}}{\mathrm{d}t}\mathcal{V}_2(z_i) &= 2 \sum_{i\in\mathbb{N}}\mathrm{Re}(a_i)|z_i|^2 + 2\sum_{i\in\mathbb{N}}z_i^\ast b_i K_0\hat{x}_0.
\end{aligned}
\end{equation*}
Using Jordan's sorted form $\mathrm{Re}(a_i)\leq \sigmaD{A_1}<0$ and applying the Lyapunov inequalities~\eqref{eq:stab0}-\eqref{eq:stab1} leads to
\begin{equation*}
\begin{aligned}
    \frac{\mathrm{d}}{\mathrm{d}t}\mathcal{V}_0(\hat{x}_0,e_0) &\leq -2\alpha\delta  \begin{bmatrix} \hat{x}_0 \\ e_0 \end{bmatrix}^\ast \!\! P_0 \begin{bmatrix} \hat{x}_0 \\ e_0 \end{bmatrix} +  2\alpha  \begin{bmatrix} \hat{x}_0 \\ e_0 \end{bmatrix}^\ast \!\! P_0 \begin{bmatrix} -G_0 \\ G_0 \end{bmatrix} \left(C_1 e_1 +   \sum_{i\in\mathbb{N}}c_iz_i\right),\\
    \frac{\mathrm{d}}{\mathrm{d}t}\mathcal{V}_1(\hat{x}_1,e_1) &\leq -2\delta\begin{bmatrix} \hat{x}_1 \\ e_1 \end{bmatrix}^\ast \begin{bmatrix}
        \beta P_1 & 0 \\ 0 & \gamma P_1
    \end{bmatrix} \begin{bmatrix} \hat{x}_1 \\ e_1 \end{bmatrix} + 2\beta \hat{x}_1^\ast P_1B_1K_0\hat{x}_0,\\
    \frac{\mathrm{d}}{\mathrm{d}t}\mathcal{V}_2(z_i) &\leq -2|\sigmaD{A_1}| \lVert (z_i)\rVert^2 + 2\sum_{i\in\mathbb{N}}z_i^\ast b_i K_0\hat{x}_0.
\end{aligned}
\end{equation*}
where the scalar $\delta$ satisfies~\eqref{eq:stab1}.
The three quadratic terms are all negative. 
The three crossed terms are then distributed on the quadratic terms. 
Applying Young's inequality gives rise to the upper bounds
\begin{align*}
    2\begin{bmatrix} \hat{x}_0 \\ e_0 \end{bmatrix}^\ast \!\!P_0\begin{bmatrix} -G_0 \\ G_0 \end{bmatrix} \left(C_1e_1 +\sum_{i\in\mathbb{N}} c_iz_i\right) &\leq 
    \frac{\delta}{2} 
    \begin{bmatrix} \hat{x}_0 \\ e_0 \end{bmatrix}^\ast \!\!P_0 \begin{bmatrix} \hat{x}_0 \\ e_0 \end{bmatrix} 
    + \frac{4}{\delta} 
    \left\lvert P_0^{\frac{1}{2}}\begin{bmatrix} -G_0 \\ G_0 \end{bmatrix}C_1e_1 \right\rvert^2
    + \frac{4}{\delta} 
    \left\lvert P_0^{\frac{1}{2}}\begin{bmatrix} -G_0 \\ G_0 \end{bmatrix} \sum_{i\in\mathbb{N}}c_iz_i \right\rvert^2,\\
    2 \hat{x}_1^\ast P_1B_1K_0 \hat{x}_0 &\leq \delta \hat{x}_1^\ast P_1 \hat{x}_1 + \frac{1}{\delta}\left\lvert P_1^{\frac{1}{2}}B_1K_0\hat{x}_0\right\rvert^2,\\
    2\sum_{i\in\mathbb{N}}z_i^\ast b_iK_0\hat{x}_0 &\leq (|\sigmaD{A_1}|-\delta) \lVert (z_i)\rVert^2 + \frac{1}{|\sigmaD{A_1}|-\delta} \sum_{i\in\mathbb{N}}  |b_iK_0\hat{x}_0|^2
\end{align*}
for any $\delta$ such that $|\sigmaD{A_1}|-\delta>0$ and $\delta>0$, as it happens the scalar $\delta$ satisfying~\eqref{eq:stab1}.
By collecting the above inequalities, the Lyapunov derivatives~$\frac{\mathrm{d}\mathcal{V}}{\mathrm{d}t}$ is bounded by
\begin{equation}\label{eq:calculation}
\begin{aligned}
    \frac{\mathrm{d}}{\mathrm{d}t}\mathcal{V}(\hat{x}_0,e_0,\hat{x}_1,e_1,z_i) \leq\;&  
    -\frac{3}{2}\alpha\delta 
    \begin{bmatrix} \hat{x}_0 \\ e_0 \end{bmatrix}^\ast\!\! P_0 \begin{bmatrix} \hat{x}_0 \\ e_0 \end{bmatrix} + \left(\frac{\beta}{\delta} \sigmaU{K_0^\ast B_1^\ast P_1 B_1K_0} + \frac{1}{|\sigmaD{A_1}|-\delta}\sum_{i\in\mathbb{N}}\sigmaU{K_0^\ast b_i^\ast b_iK_0}\right) \lvert \hat{x}_0\rvert^2 \\
    & -\delta\begin{bmatrix} \hat{x}_1 \\ e_1 \end{bmatrix}^\ast \begin{bmatrix}
        \beta P_1 & 0 \\ 0 & 2\gamma P_1
    \end{bmatrix} \begin{bmatrix} \hat{x}_1 \\ e_1 \end{bmatrix} + 4\frac{\alpha}{\delta} 
    \sigmaU{C_1^\ast \begin{bsmallmatrix}-G_0\\G_0\end{bsmallmatrix}^\ast P_0\begin{bsmallmatrix}-G_0\\G_0\end{bsmallmatrix} C_1} \lvert e_1 \rvert^2\\
    & + \left(-\delta-|\sigmaD{A_1}| + 4\frac{\alpha}{\delta}
   \sum_{i\in\mathbb{N}}\sigmaU{c_i^\ast \begin{bsmallmatrix}-G_0\\G_0\end{bsmallmatrix}^\ast P_0\begin{bsmallmatrix}-G_0\\G_0\end{bsmallmatrix} c_i}\right) \lVert (z_i)\rVert^2.
\end{aligned}
\end{equation}

In the next paragraph, these calculations are used to demonstrate stability through the Lyapunov theorem. The weighting $\alpha,\beta,\gamma$ are selected in a way to obtain a negative upper bound and satisfy $\frac{\mathrm{d}\mathcal{V}}{\mathrm{d}t}\leq-\delta \mathcal{V}$.

\subsection{Stability analysis}

In this subsection, we prove that our finite-dimensional controller~\eqref{eq:matrix1} stabilizes system~\eqref{eq:system}. More precisely, based on the previous Lyapunov functional, if the ratio between the upper bound of the input and output operators and the square of the lower bound of the state operator is bounded, then the closed-loop system~\eqref{eq:system}-\eqref{eq:cont} is exponentially stable.
\begin{thm}\label{thm1}
    Let $\delta>0$.
    Under Assumptions~\ref{assum1} and~\ref{assum2}, if there exists an integer $n\geq n_0+n_1$ such that 
    \begin{equation}\label{eq:order1}
        \varrho_n:=\frac{16\sum_{i\in\mathbb{N}}\sigmaU{K_0^\ast b_i^\ast b_iK_0}\sum_{i\in\mathbb{N}}\sigmaU{c_i^\ast \begin{bsmallmatrix}-G_0\\G_0\end{bsmallmatrix}^\ast P_0\begin{bsmallmatrix}-G_0\\G_0\end{bsmallmatrix} c_i}}{\delta^2\sigmaD{P_0}(|\sigmaD{A_1}|-\delta)|\sigmaD{A_1}|} \leq 1,
    \end{equation}
    where $K_0,G_0,P_0$ are given by~\eqref{eq:stab0},
    then, the closed-loop system~\eqref{eq:system}-\eqref{eq:cont} with $(K,L,M,N)$ in~\eqref{eq:matrix1} is exponentially stable with decay rate $\delta$.
\end{thm}
\begin{proof}
    Take the Lyapunov functional~$\mathcal{V}$ defined by~\eqref{eq:lyap} with the weights 
    \begin{equation*}
    \alpha \!=\! 
    \frac{4}{\delta} 
    \frac{\sum_{i\in\mathbb{N}}\sigmaU{K_0^\ast b_i^\ast b_iK_0}}{\sigmaD{P_0}(|\sigmaD{A_1}|-\delta)},\quad
    \beta \!=\! \frac{\delta\sum_{i\in\mathbb{N}}\sigmaU{K_0^\ast b_i^\ast b_iK_0}}{(|\sigmaD{A_1}|-\delta)\sigmaU{K_0^\ast B_1^\ast P_1B_1K_0}},\quad
    \gamma \!=\! 
    \frac{4\alpha \sigmaU{C_1^\ast \begin{bsmallmatrix}-G_0\\G_0\end{bsmallmatrix}^\ast P_0\begin{bsmallmatrix}-G_0\\G_0\end{bsmallmatrix} C_1}}{\delta^2 \sigmaD{P_1}}, 
    \end{equation*}
    where the series introduced here exist by boundedness of operators $\mathcal{B}$ and $\mathcal{C}$.
    This Lyapunov functional is framed by
    \begin{equation*}
        \mathrm{min}(\alpha\sigmaD{P_0},\beta\sigmaD{P_1},\gamma\sigmaD{P_1},1) 
        \leq \frac{\mathcal{V}(\hat{x}_0,e_0,\hat{x}_1,e_1,z_i)}{|\hat{x}_0|^2+|e_0|^2+|\hat{x}_1|^2+|e_1|^2+\lVert (z_i)\rVert^2} 
        \leq \alpha\sigmaU{P_0}+\beta\sigmaU{P_1}+\gamma\sigmaU{P_1}+1.
    \end{equation*}
    Under Assumptions~\ref{assum1} and~\ref{assum2}, the inequality~\eqref{eq:calculation} with the scalars $(\alpha,\beta,\gamma)$ above rewrites as follows
    \begin{equation*}
    \begin{aligned}
    \frac{\mathrm{d}}{\mathrm{d}t}\mathcal{V}(\hat{x}_0,e_0,\hat{x}_1,e_1,z_i) &\leq 
    -\delta \underbrace{\alpha
    \begin{bmatrix} \hat{x}_0 \\ e_0 \end{bmatrix}^\ast\!\! P_0 \begin{bmatrix} \hat{x}_0 \\ e_0 \end{bmatrix}}_{\mathcal{V}_0(\hat{x}_0,e_0)} + \left( -\alpha \frac{\delta \sigmaD{P_0}}{2} + 2\frac{\sum_{i\in\mathbb{N}}\sigmaU{K_0^\ast b_i^\ast b_iK_0}}{|\sigmaD{A_1}|-\delta}\right) \lvert \hat{x}_0\rvert^2 \\
    &\quad -\delta \underbrace{\begin{bmatrix} \hat{x}_1 \\ e_1 \end{bmatrix}^\ast \begin{bmatrix}
        \beta P_1 & 0 \\ 0 & \gamma P_1
    \end{bmatrix} \begin{bmatrix} \hat{x}_1 \\ e_1 \end{bmatrix} }_{\mathcal{V}_1(\hat{x}_1,e_1)} + \left( -\gamma\delta\sigmaD{P_1} + 4\frac{\alpha}{\delta} 
    \sigmaU{C_1^\ast \begin{bsmallmatrix}-G_0\\G_0\end{bsmallmatrix}^\ast P_0\begin{bsmallmatrix}-G_0\\G_0\end{bsmallmatrix} C_1} \right)\lvert e_1 \rvert^2\\
    &\quad + -\delta \underbrace{\lVert (z_i)\rVert^2}_{\mathcal{V}_2(z_i)}+\left(-|\sigmaD{A_1}| + 4\frac{\alpha}{\delta}
   \sum_{i\in\mathbb{N}}\sigmaU{c_i^\ast \begin{bsmallmatrix}-G_0\\G_0\end{bsmallmatrix}^\ast P_0\begin{bsmallmatrix}-G_0\\G_0\end{bsmallmatrix} c_i}\right) \lVert (z_i)\rVert^2,\\
    &\leq
        -\delta \mathcal{V}(\hat{x}_0,e_0,\hat{x}_1,e_1,z_i) 
        - (1-\varrho_n)|\sigmaD{A_1}| \lVert (z_i)\rVert^2.
        \end{aligned}
    \end{equation*}
    Note that $\alpha$ and $\gamma$ can also be selected larger and that $\beta$ can be selected smaller.
    Then, from inequality~\eqref{eq:order1}, we can apply the Lyapunov theorem and obtain the exponential convergence of the state $(\hat{x}_0,e_0,\hat{x}_1,e_1,z_i)$ with the decay rate~$\delta$.
    At the price of a change of variable, we prove that there exists $\kappa\geq 1$ such that
    \begin{equation*}
        |x(t)|^2 + |\hat{x}(t)|^2+ \lVert z_i(t)\rVert^2 \leq \kappa \exp(-\delta t) \left( |x(0)|^2 + |\hat{x}(0)|^2+ \lVert z_i(0)\rVert^2 \right)
    \end{equation*}
    and conclude on the exponential convergence toward the equilibrium with a decay rate of $\delta$.
\end{proof}

An interpretation of the limitations imposed by the inequality~\eqref{eq:order1} is proposed and discussed in the following.

\begin{cor}\label{cor1}
    Let $\delta>0$. Under Assumptions~\ref{assum1} and~\ref{assum2}, if the sorted eigenvalues $\{\lambda_i\}_{i\in\mathbb{N}}$ of the operator $\mathcal{A}$ satisfy $\mathrm{Re}(\lambda_i)\underset{i\to\infty}{\longrightarrow}-\infty$, then there exists a sufficiently large order $n$ such that the closed-loop system is exponentially stable with a decay rate $\delta$.
\end{cor}
\begin{proof}
    Fix matrices $P_0,K_0,G_0$ with respect to the order $n_0$, independently of~$n$. 
    The boundedness of the operators $\mathcal{B}$ and $\mathcal{C}$ guarantees the existence of a scalar $k>0$ independent of~$n$ such that
    $$\varrho_n=\frac{16\sum_{i\in\mathbb{N}}\sigmaU{K_0^\ast b_i^\ast b_iK_0}\sum_{i\in\mathbb{N}}\sigmaU{c_i^\ast \begin{bsmallmatrix}-G_0\\G_0\end{bsmallmatrix}^\ast P_0\begin{bsmallmatrix}-G_0\\G_0\end{bsmallmatrix} c_i}}{\delta^2\sigmaD{P_0}|\mathrm{Re}(\lambda_n)|\left(|\mathrm{Re}(\lambda_n)|-\delta\right)}\leq \frac{k}{|\mathrm{Re}(\lambda_n)|\left(|\mathrm{Re}(\lambda_n)|-\delta\right)}.$$ Setting the order $n$ sufficiently large order $n$ yields $\varrho_n\leq 1$ and Theorem~\ref{thm1} concludes the proof.
\end{proof}

\begin{rem}
    Note that inequality~\eqref{eq:order1} alleviates the asymptotic condition on the eigenvalues of the operator~$\mathcal{A}$.
    Indeed, it suffices to have a ratio between the upper bound of the input and output operators and the square of the lower bound of the state operator which is smaller than one. This trade-off between input-output gain and decay rate of the remaining part is not totally new, remembering that we neglect fast in front of slow dynamics in perturbation theory.
\end{rem}

 \section{Synthesis of the controller with uncertainties}

In this section, the state, input and output operators $(A_0,A_1,B_0,B_1,C_0,C_1)$ are subject to uncertainties and the matrices~$(\hat{A}_0,\hat{A}_1,\hat{B}_0,\hat{B}_1,\hat{C}_0,\hat{C}_1)$ are chosen so that the pair~$(\hat{A}_0,\hat{B}_0)$ is controllable and the pair~$(\hat{C}_0,\hat{A}_0)$ is observable. It will serve as our knowledge model. 

\begin{rem}
    There are many reasons for this extension to the uncertain case. Indeed, as soon as the eigenvalues are known to within one error, we find ourselves in this configuration. This is often the case when we need to solve the characteristic equation numerically. Moreover, we can also imagine dealing with robustness problems such as the case of a reaction-diffusion equation with a constant but uncertain reaction coefficient. In the longer term, it would also be worth looking at the question of uncertain eigenstructures.
\end{rem}

\subsection{Controller data}

In presence of uncertainties, the same control strategy is adopted by replacing the exact model with the approximate model. More precisely, we consider
\begin{equation}\label{eq:cont2}
\left\{
\begin{aligned}
    \frac{\mathrm{d}}{\mathrm{d}t}\begin{bsmallmatrix}\hat{x}_0(t)\\\hat{x}_1(t)\end{bsmallmatrix} &= \begin{bsmallmatrix} \hat{A}_0 & 0\\ 0 & \hat{A}_1 \end{bsmallmatrix} \begin{bsmallmatrix}\hat{x}_0(t)\\\hat{x}_1(t)\end{bsmallmatrix} + \begin{bsmallmatrix}\hat{B}_0\\\hat{B}_1\end{bsmallmatrix} u(t) + \begin{bsmallmatrix}G_0\\0\end{bsmallmatrix} (\hat{y}(t)-y(t)) \\
    \hat{y}(t) &= \begin{bsmallmatrix} \hat{C}_0 & \hat{C}_1 \end{bsmallmatrix} \begin{bsmallmatrix}\hat{x}_0(t)\\\hat{x}_1(t)\end{bsmallmatrix} , \\
    u(t) &= K_0 {x}_0(t),
\end{aligned}
\right.
\end{equation}
for some suitable control gains $K_0\in\mathbb{R}^{n_u\times n_0}$ and observation gains $G_0\in\mathbb{R}^{n_0\times n_y}$ to be fixed later. Hence, the above finite-dimensional controller can be written as~\eqref{eq:cont} with matrices 
\begin{equation}\label{eq:matrix2}
\begin{aligned}
    L &= \begin{bmatrix}
        \hat{A}_0+G_0\hat{C}_0 & G_0\hat{C}_1\\
        0 & \hat{A}_1
    \end{bmatrix},\quad
    M = \begin{bmatrix}
        -G_0\\ 0
    \end{bmatrix},\quad
    N = \begin{bmatrix}
        \hat{B}_0\\ \hat{B}_1
    \end{bmatrix},\quad
    K = \begin{bmatrix}
        K_0 & 0
    \end{bmatrix}.
\end{aligned}
\end{equation}

Here, the dynamics of the reconstructed state $(\hat{x}_0,\hat{x}_1)$ and the estimation error state $(e_0,e_1)=(x_0-\hat{x}_0,x_1-\hat{x}_1)$ of the observer~\eqref{eq:cont2} satisfy
\begin{equation}\label{eq:systemCL}
\left\{
\begin{aligned}
    \frac{\mathrm{d}}{\mathrm{d}t}\begin{bmatrix}\hat{x}_0\\e_0\end{bmatrix} &= (\hat{F}_0+\tilde{F}_0)\begin{bmatrix}\hat{x}_0\\e_0\end{bmatrix} + \begin{bmatrix} -G_0  \\ G_0 \end{bmatrix}C_1e_1 + \begin{bmatrix} -G_0 \\ G_0 \end{bmatrix} \tilde{C}_1\hat{x}_1 + \begin{bmatrix} -G_0 \\ G_0 \end{bmatrix} \sum_{i\in\mathbb{N}}c_iz_i.,\\
    \frac{\mathrm{d}}{\mathrm{d}t}\begin{bmatrix}\hat{x}_1\\e_1\end{bmatrix} &= \left(\begin{bmatrix}A_1&0\\0&A_1\end{bmatrix} + \begin{bmatrix}-\tilde{A}_1&0\\\tilde{A}_1&0\end{bmatrix}\right)\begin{bmatrix}\hat{x}_1\\e_1\end{bmatrix} + \begin{bmatrix}\hat{B}_1K_0\\\tilde{B}_1K_0\end{bmatrix}\hat{x}_0,
\end{aligned}
\right.
\end{equation}
where
\begin{equation*}
\begin{aligned}
    \hat{F}_0 &= \begin{bmatrix} \hat{A}_0\!+\!\hat{B}_0K_0 & -G_0\hat{C}_0 \\ 0 & \hat{A}_0 \!+\! G_0\hat{C}_0 \end{bmatrix}, \qquad \tilde{F}_0 = \begin{bmatrix} -G_0\tilde{C}_0 & -G_0\tilde{C}_0 \\ \tilde{A}_0 \!+\! \tilde{B}_0K_0 \!+\! G_0\tilde{C}_0  & \tilde{A}_0 \!+\! G_0\tilde{C}_0 \end{bmatrix},\\
    \tilde{A}_0 &= A_0-\hat{A}_0,\; \tilde{A}_1 = A_1-\hat{A}_1,\; \tilde{B}_0 = B_0-\hat{B}_0,\;\tilde{B}_1 = B_1-\hat{B}_1,\; \tilde{C}_0=C_0 - \hat{C}_0,\; \tilde{C}_1 = C_1-\hat{C}_1.
\end{aligned}
\end{equation*}

Assuming that the pairs $(\hat{A}_0,\hat{B}_0)$ and $(\hat{C}_0,\hat{A}_0)$ are controllable and observable, the gains $K_0$ and $G_0$ can be selected such that there exists a symmetric positive matrix $P_0\in\mathbb{R}^{n_0\times n_0}$ which satisfy the Lyapunov inequality
\begin{equation}\label{eq:stab4}
    P_0 \hat{F}_0 + \hat{F}_0^\ast P_0 \prec -2\delta P_0.
\end{equation}
In the next paragraph, Lyapunov analysis is performed to provide the stability result.

\subsection{Stability analysis}


In this subsection, we prove that our approximated finite-dimensional controller~\eqref{eq:matrix2} stabilizes system~\eqref{eq:system}. More precisely, based on the Lyapunov theorem, if the approximation error is sufficiently small and if the ratio between the upper bound of the input and output operators and the square of the lower bound of the state operator is bounded, then the closed-loop system is exponentially stable.

\begin{thm}\label{thm2}
    Let $\delta>0$ and assume that Assumption~\ref{assum1} holds and that the pairs $(\hat{A}_0,\hat{B}_0)$ and $(\hat{C}_0,\hat{A}_0)$ are controllable and observable.  If there exists an integer $n\geq n_0+n_1$ such that 
    \begin{equation}\label{eq:order2}
        \hat{\varrho}_n:=\frac{16\sum_{i\in\mathbb{N}}\sigmaU{K_0^\ast b_i^\ast b_iK_0}\sum_{i\in\mathbb{N}}\sigmaU{c_i^\ast \begin{bsmallmatrix}-G_0\\G_0\end{bsmallmatrix}^\ast P_0\begin{bsmallmatrix}-G_0\\G_0\end{bsmallmatrix} c_i}}{\delta^2\sigmaD{P_0}(|\sigmaD{A_1}|-\delta)|\sigmaD{A_1}|} \leq 1,
    \end{equation}
    and if the model uncertainties satisfy 
    \begin{equation}\label{eq:eta}
        \eta:=\mathrm{max}(\eta_0,\eta_1,\eta_2)<\delta,
    \end{equation}
    where
    \begin{equation}
    \left\{
    \begin{aligned}
        \eta_0 &=  2 \sqrt{\sigmaU{\tilde{F}_0^\ast\tilde{F}_0}} + \sigmaU{\tilde{C}_1^\ast \begin{bsmallmatrix} -G_0 \\ G_0\end{bsmallmatrix}^\ast P_0\begin{bsmallmatrix} -G_0 \\G_0 \end{bsmallmatrix} \tilde{C}_1} + \frac{\gamma \sigmaU{K_0^\ast \tilde{B}_1^\ast P_1\tilde{B}_1K_0} }{\alpha\sigmaD{P_0}},\\
        \eta_1 &= \left(2 + \frac{\gamma}{\beta}\right) \sqrt{\sigmaU{\tilde{A}_1^\ast\tilde{A}_1}} + \frac{\alpha \sigmaU{\tilde{C}_1^\ast \begin{bsmallmatrix} -G_0 \\ G_0\end{bsmallmatrix}^\ast P_0\begin{bsmallmatrix} -G_0 \\G_0 \end{bsmallmatrix} \tilde{C}_1}}{\beta\sigmaD{P_1}},\\
        \eta_2 &= \sigmaU{K_0^\ast \tilde{B}_1^\ast P_1\tilde{B}_1K_0} + \sqrt{\sigmaU{\tilde{A}_1^\ast\tilde{A}_1}}.
    \end{aligned}
    \right.
    \end{equation}
    with $K_0,G_0,P_0,P_1$ are given by~\eqref{eq:stab1} and \eqref{eq:stab4} and $\alpha,\beta,\gamma$ in~\eqref{eq:abg}, then, the closed-loop system~\eqref{eq:system}-\eqref{eq:cont} with $(K,L,M,N)$ in~\eqref{eq:matrix2} is exponentially stable with the decay rate $\delta-\eta$.
\end{thm}

\begin{proof}
    Consider the Lyapunov functional defined by~\eqref{eq:lyap} with matrices $P_0,P_1$ fixed in accordance with~\eqref{eq:stab1} and \eqref{eq:stab4}. Using Jordan's sorted form $\mathrm{Re}(a_i)\leq \sigmaD{A_1}<0$ and applying the Lyapunov inequalities~\eqref{eq:stab1} and \eqref{eq:stab4}, we obtain 
    \begin{equation*}
    \begin{aligned}
        \frac{\mathrm{d}}{\mathrm{d}t}\mathcal{V}_0(\hat{x}_0,e_0) &\leq -2\alpha\delta  \begin{bmatrix} \hat{x}_0 \\ e_0 \end{bmatrix}^\ast \!\! P_0 \begin{bmatrix} \hat{x}_0 \\ e_0 \end{bmatrix}  +  2\alpha  \begin{bmatrix} \hat{x}_0 \\ e_0 \end{bmatrix}^\ast \!\! P_0 \tilde{F}_0 \begin{bmatrix} \hat{x}_0 \\ e_0 \end{bmatrix} 
        +  2\alpha  \begin{bmatrix} \hat{x}_0 \\ e_0 \end{bmatrix}^\ast \!\! P_0 \begin{bmatrix} -G_0  \\ G_0\end{bmatrix}\left( C_1 e_1 + \tilde{C}_1 \hat{x}_1 + \sum_{i\in\mathbb{N}}c_i z_i\right),\\
        \frac{\mathrm{d}}{\mathrm{d}t}\mathcal{V}_1(\hat{x}_1,e_1) &\leq -2\delta\begin{bmatrix} \hat{x}_1 \\ e_1 \end{bmatrix}^\ast \begin{bmatrix}
            \beta P_1 & 0 \\ 0 & \gamma P_1
        \end{bmatrix} \begin{bmatrix} \hat{x}_1 \\ e_1 \end{bmatrix} + 2 \begin{bmatrix} \hat{x}_1 \\ e_1 \end{bmatrix}^\ast \begin{bmatrix}
            -\beta P_1\tilde{A}_1 \\ \gamma P_1\tilde{A}_1
        \end{bmatrix} \hat{x}_1
        + 2 \begin{bmatrix} \hat{x}_1 \\ e_1 \end{bmatrix}^\ast \begin{bmatrix}
            \beta P_1\hat{B}_1K_0\\ \gamma P_1\tilde{B}_1K_0
        \end{bmatrix}\hat{x}_0,\\
        \frac{\mathrm{d}\mathcal{V}_2}{\mathrm{d}t}(z_i) &\leq -2|\sigmaD{A_1}|\lVert (z_i)\rVert^2 + 2\sum_{i\in\mathbb{N}}z_i^\ast b_i K_0\hat{x}_0.
    \end{aligned}
    \end{equation*}
    As in the previous section (see proof of Theorem~\ref{thm1}), the weights~$(\alpha,\beta,\gamma)$ are selected in the adequate manner
    \begin{equation}\label{eq:abg}
        \alpha \!=\! \frac{4}{\delta} \frac{\sum_{i\in\mathbb{N}}\sigmaU{K_0^\ast b_i^\ast b_iK_0}}{\sigmaD{P_0}(|\sigmaD{A_1}|-\delta)},\quad
        \beta \!=\! \frac{\delta\sum_{i\in\mathbb{N}}\sigmaU{K_0^\ast b_i^\ast b_iK_0}}{(|\sigmaD{A_1}|-\delta)\sigmaU{K_0^\ast \hat{B}_1^\ast P_1 \hat{B}_1K_0}},\quad
        \gamma \!=\! \frac{4\alpha\sigmaU{C_1^\ast \begin{bsmallmatrix}-G_0\\G_0\end{bsmallmatrix}^\ast\!\! P_0\begin{bsmallmatrix}-G_0\\G_0\end{bsmallmatrix} C_1}}{\delta^2 \sigmaD{P_1}},
    \end{equation}
    to obtain
    \begin{equation*}
    \begin{aligned}
         \frac{\mathrm{d}}{\mathrm{d}t}\mathcal{V}(\hat{x}_0,e_0,\hat{x}_1,e_1,z_i) \leq\;& \left(-\delta+2\sqrt{\sigmaU{\tilde{F}_0^\ast\tilde{F}_0}}\right) \alpha \begin{bmatrix} \hat{x}_0 \\ e_0 \end{bmatrix}^\ast\!\!\! P_0 \begin{bmatrix} \hat{x}_0 \\ e_0 \end{bmatrix}
         +\begin{bmatrix} \hat{x}_1 \\ e_1 \end{bmatrix}^\ast \begin{bmatrix} \beta\left(-\delta+2\sqrt{\sigmaU{\tilde{A}_1^\ast\tilde{A}_1}}\right) P_1 & 0 \\ 0 & -2\gamma\delta P_1 \end{bmatrix} \begin{bmatrix} \hat{x}_1 \\ e_1 \end{bmatrix} \\
         & +\left(-\delta -|\sigmaD{A_1}| + 4\frac{\alpha}{\delta}\sum_{i\in\mathbb{N}}\sigmaU{c_i^\ast \begin{bsmallmatrix}-G_0\\G_0\end{bsmallmatrix}^\ast P_0\begin{bsmallmatrix}-G_0\\G_0\end{bsmallmatrix} c_i}\right) \lVert (z_i)\rVert^2\\
        & + 2\alpha \begin{bmatrix} \hat{x}_0 \\ e_0 \end{bmatrix}^\ast P_0\begin{bmatrix} -G_0 \\ G_0 \end{bmatrix} \tilde{C}_1 \hat{x}_1 
        + 2\gamma e_1^\ast P_1(\tilde{A}_1\hat{x}_1+\tilde{B}_1K_0\hat{x}_0).
    \end{aligned}
    \end{equation*}
    Then, we apply Young inequality and matrix supreme norms to the additional two crossed terms as follows
    \begin{align*}
        2 \alpha \begin{bmatrix} \hat{x}_0 \\ e_0 \end{bmatrix}^\ast P_0\begin{bmatrix} -G_0 \\ G_0 \end{bmatrix} \tilde{C}_1 \hat{x}_1 &\leq \alpha \sigmaU{\tilde{C}_1^\ast \begin{bsmallmatrix} -G_0 \\ G_0\end{bsmallmatrix}^\ast P_0\begin{bsmallmatrix} -G_0 \\G_0 \end{bsmallmatrix} \tilde{C}_1} \left( \begin{bmatrix} \hat{x}_0 \\ e_0 \end{bmatrix}^\ast P_0 \begin{bmatrix} \hat{x}_0 \\ e_0 \end{bmatrix} + |\hat{x}_1|^2 \right),\\
        2\gamma e_1^\ast P_1(\tilde{A}_1\hat{x}_1+\tilde{B}_1K_0\hat{x}_0) &\leq
        \gamma \sqrt{\sigmaU{\tilde{A}_1^\ast\tilde{A}_1}}\left( e_1^\ast P_1 e_1 + \hat{x}_1^\ast P_1 \hat{x}_1 \right) + \gamma \sigmaU{K_0^\ast\tilde{B}_1^\ast P_1\tilde{B}_1K_0}\left(e_1^\ast P_1 e_1 + \hat{x}_0^\ast \hat{x}_0\right),
    \end{align*}
    to get to
    \begin{equation*}
        \frac{\mathrm{d}}{\mathrm{d}t}\mathcal{V}(\hat{x}_0,e_0,\hat{x}_1,e_1,z_i) \leq
        (-\delta + \eta )\mathcal{V}(\hat{x}_0,e_0,\hat{x}_1,e_1,z_i) - (1-\hat{\varrho}_n)|\sigmaD{A_1}|\lVert (z_i)\rVert^2,
    \end{equation*}
    where the substantial error $\eta=\max(\eta_0,\eta_1,\eta_2)$. Then, from inequality~\eqref{eq:order2}, we can apply the Lyapunov theorem and obtain the exponential convergence of the state $(\hat{x}_0,e_0,\hat{x}_1,e_1,z_i)$ with the decay rate $\delta-\eta$. By variable changes, we conclude on the exponential stability of system~\eqref{eq:system}-\eqref{eq:cont} in terms of variables $(x,\hat{x},z)$ with the same decay rate.
\end{proof}

An interpretation of the inequalities~\eqref{eq:order2} and~\eqref{eq:eta} is proposed and discussed in the following.

\begin{cor}\label{cor2}
    Let 
    $\delta>0$. Assume that, for any $n\in\mathbb{N}$ and $\varepsilon>0$, there exists matrices $(\hat{A}_0,\hat{A}_1,\hat{B}_0,\hat{B}_1,\hat{C}_0,\hat{C}_1)$ such that 
    \begin{equation}\label{eq:etab}
        \mathrm{max}(\sigmaU{\tilde{A}_0^\ast\tilde{A}_0},\sigmaU{\tilde{A}_1^\ast\tilde{A}_1},\sigmaU{\tilde{B}_0^\ast\tilde{B}_0},\sigmaU{\tilde{B}_1^\ast\tilde{B}_1},\sigmaU{\tilde{C}_0^\ast\tilde{C}_0},\sigmaU{\tilde{C}_1^\ast\tilde{C}_1})\leq \varepsilon.
    \end{equation}
    Assume also that Assumption~\ref{assum1} holds and that the pairs $(\hat{A}_0,\hat{B}_0)$ and $(\hat{C}_0,\hat{A}_0)$ are controllable and observable. If the sorted eigenvalues $\{\lambda_i\}_{i\in\mathbb{N}}$ of the operator $\mathcal{A}$ satisfy $\mathrm{Re}(\lambda_i)\underset{i\to\infty}{\longrightarrow}-\infty$, then there exists a sufficiently large order $n$ such that the closed-loop system is exponentially stable.
\end{cor}

\begin{proof}
    The satisfaction of~\eqref{eq:etab} guarantees that $\eta<\delta$ by manipulating matrix norm combinations .
    As for Corollary~\ref{cor1}, the boundedness of operators $\mathcal{B}$ and $\mathcal{C}$ and $\mathrm{Re}(\lambda_i)\underset{i\to\infty}{\longrightarrow}-\infty$ imply that the inequality~\eqref{eq:order2} holds.
    The application of Theorem~\ref{thm2} concludes the proof.
\end{proof}

\begin{rem}
    Two additional conditions appear in Theorem~\ref{thm2} (Corollary~\ref{cor2}) compared to Theorem~\ref{thm1} (Corollary~\ref{cor1}). First, the pairs $(\hat{A}_0,\hat{B}_0)$ and $(\hat{C}_0,\hat{A}_0)$ must be controllable and observable, respectively. Second, we need to be able to make the model errors $\sigmaU{\tilde{A}_0^\ast\tilde{A}_0}$, $\sigmaU{\tilde{A}_1^\ast\tilde{A}_1}$, $\sigmaU{\tilde{B}_0^\ast\tilde{B}_0}$, $\sigmaU{\tilde{B}_1^\ast\tilde{B}_1}$, $\sigmaU{\tilde{C}_0^\ast\tilde{C}_0}$ and $\sigmaU{\tilde{C}_1^\ast\tilde{C}_1}$ as small as needed. Fortunately, for linear ODE-PDE coupled systems, classical numerical methods such as Padé rational approximation~\cite{Baker96} or quasi-spectral approximation~\cite{Gottlieb77} can be used to fulfill these requirements. When used on the right class of systems, they allow to approximate as well as desired the eigenvalues $(s_k)_{k\in\mathbb{N}}$ in a bounded region of the complex plane. Using the analytic expression of the eigenvectors $(v_k)_{k\in\mathbb{N}}$ with respect to the eigenvalues (see~\eqref{eq:xik1} or~\eqref{eq:xik2}), the input, output and state approximated matrices are obtained while keeping the eigenstructure given by~\eqref{eq:ABC} with a sufficiently small model mismatch.
\end{rem}

\section{Numerical applications}

The stabilization of three high-dimensional systems are presented using Matlab, with a code is available online\footnote{\textit{https://github.com/mat-bajo/control-id}}.

\subsection{Toy example}

\begin{figure}
    \begin{subfigure}{0.45\linewidth}
    \centering
    \includegraphics[width=6cm]{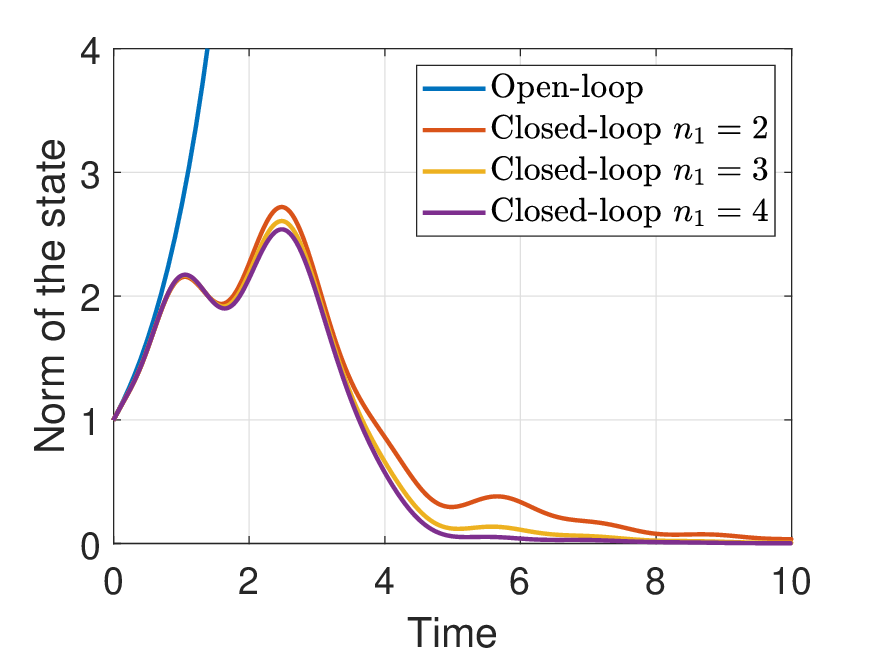}
    \caption{Closed-loop system for several orders $n_1$.}
    \label{fig1}
    \end{subfigure}
    \begin{subfigure}{0.45\linewidth}
    \centering
    \includegraphics[width=6cm]{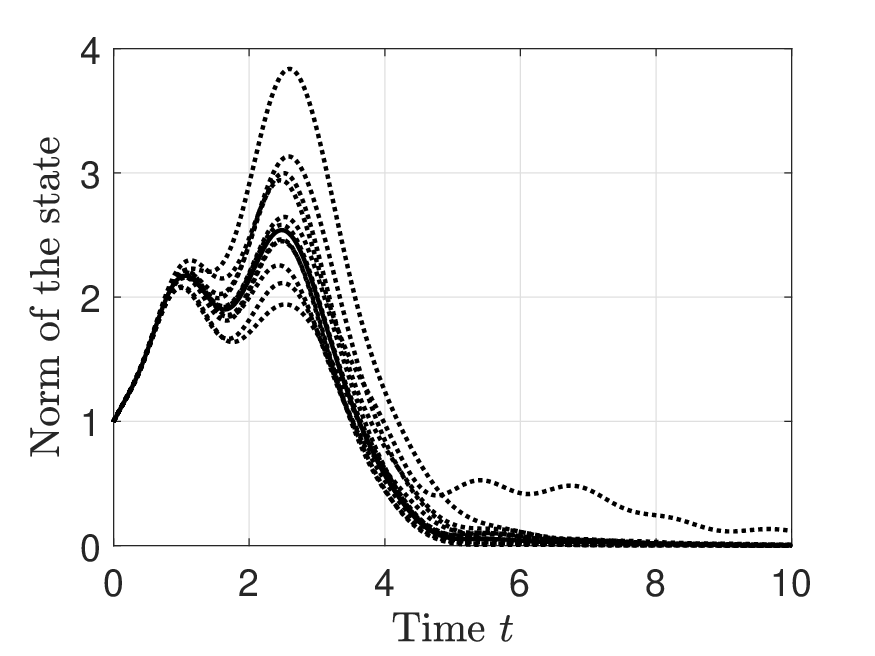}
    \caption{Closed-loop system for $n_1=4$ and several uncertainties.}
    \label{fig2}
    \end{subfigure}
\end{figure}

Consider the following system~\eqref{eq:system} with matrices
\begin{equation}\label{ex:1}
A_0 = \begin{bsmallmatrix}0.5 & 2\\-2 & 0.5\end{bsmallmatrix},\; A_1 = \mathrm{diag}(-1,\dots,-n_1^2),\; B_0=\begin{bsmallmatrix}1\\1\end{bsmallmatrix},\;B_1=\begin{bsmallmatrix}1\\\vdots\\1\end{bsmallmatrix},\;C_0^\top=\begin{bsmallmatrix}1\\1\end{bsmallmatrix},\;C_1^\top=\begin{bsmallmatrix}1\\\vdots\\1\end{bsmallmatrix},\; a_i = -(n_1+i)^2,\; b_i = 1,\; c_i = 1.
\end{equation}
where $n_0=2$ and $n_1\in\mathbb{N}$. 

Designing a control gain $K_0$ such that 
\begin{equation}\label{eq:design}
    \sigma(A_0+B_0K_0) = -0.5\pm i,\quad \sigma(A_0+C_0G_0) = -0.5\pm i.
\end{equation}
the inequality~\eqref{eq:order1} is satisfied for orders $n_1\geq 2$.

In Fig.~\ref{fig1}, we plot the norm of the solution of the closed loop system for orders $n_1=\{2,3,4\}$. We confirm the system stabilizes from the order $2$, as stated in Theorem~\ref{thm1} . In Fig.~\ref{fig2}, we illustrate the uncertainty case. Computing the controller with uncertainties on $A_0,A_1,B_0,B_1,C_0,C_1$, to which we add a uniformly distributed random numbers in $[-0.1,0.1]$ on each coefficients. It turns out that the stability properties are preserved as stated in Theorem~\ref{thm2}.

\subsection{ODE-transport interconnection}

Consider system~\eqref{eq:system1} with  $A=1$, $B= -2$, $C = 1$, $B_u=1$, $C_y=1$ and $h=0.7$.

The generalized characteristic values and vectors being difficult to given analytically, we opt for a numerical method based on Padé's approximation of the transport PDE part, associated to the transfer function $H(s)=e^{-hs}$~\cite{Baker96}. Several ways of constructing the approximated model are then possible and lead to the same result. For a given order $N\in\mathbb{N}$, the Padé approximation enables to estimate the eigenvalues~$(s_k)_{k\in\{1,\dots,N\}}$ and the eigenvectors~$(v_k)_{k\in\{1,\dots,N\}}$ applying the formula~\eqref{eq:xik1}. The matrices $(\hat{A}_0,\hat{A}_1,\hat{B}_0,\hat{B}_1,\hat{C}_0,\hat{C}_1)$ can then be computed through~\eqref{eq:ABC}. Another computational way consists in using the rational Padé approximated transfer function $H_N(s)= C_N (sI_N-A_N )^{-1}B_N + D_N$ to build the following approximated model $\hat{A}_N = 
    \begin{bsmallmatrix}
        A+BD_NC & BC_N\\
        B_NC & A_N
    \end{bsmallmatrix},\quad \hat{B}_{u,N} = 
    \begin{bsmallmatrix}
        B_u\\0
    \end{bsmallmatrix},\quad \hat{C}_{y,N}^\top = 
    \begin{bsmallmatrix}
        C_y\\0
    \end{bsmallmatrix},
$
which can be set in the complex Jordan form and give rise to the approximated matrices $(\hat{A}_0,\hat{A}_1,\hat{B}_0,\hat{B}_1,\hat{C}_0,\hat{C}_1)$.\\
Let $\delta = -\frac{1}{2}$. Once in the Jordan form, we have
\begin{equation}
    \hat{A}_0 = \begin{bsmallmatrix} 0.1863 - 1.5555i &0\\0& 0.1863 + 1.5555i\end{bsmallmatrix},\; \hat{B}_0 = \begin{bsmallmatrix}0.1239 + 0.3596i
   \\0.1239 - 0.3596i\end{bsmallmatrix},\; \hat{C}_0^\top = \begin{bsmallmatrix} 2.2437 - 0.1003i \\  2.2437 + 0.1003i \end{bsmallmatrix}.
\end{equation}
where $n_0=2$. Since the inequality~\eqref{eq:order2} seems to be satisfied for all orders $n_1\geq 0$, we impose $n_1=0$ and matrices $A_1,B_1,C_1$ are not taken into account.  At the order $N=10$, the Padé approximation introduces a model error~$\mathrm{max}(\sigmaU{\tilde{A}_0^\ast\tilde{A}_0},\sigmaU{\tilde{B}_0^\ast\tilde{B}_0},\sigmaU{\tilde{C}_0^\ast\tilde{C}_0})$ which is smaller than the numerical accuracy on Matlab.

With the initial condition $x(\tau)=1$, for all $\tau\leq 0$, the instantaneous state of the open-loop system is plotted in blue on Fig~\ref{fig3}. It is an unstable system. Designing the controller according to~\eqref{eq:design}, we stabilize the system as proved in Corollary~\ref{cor2}. The solution is plotted in red on Fig~\ref{fig3} and converges exponentially fast to zero with a decay rate of $-\frac{1}{2}$. We succeeded in stabilizing an unstable time-delay system with a finite-dimensional controller.

\begin{figure}
    \begin{subfigure}{0.45\linewidth}
    \centering
    \includegraphics[width=6cm]{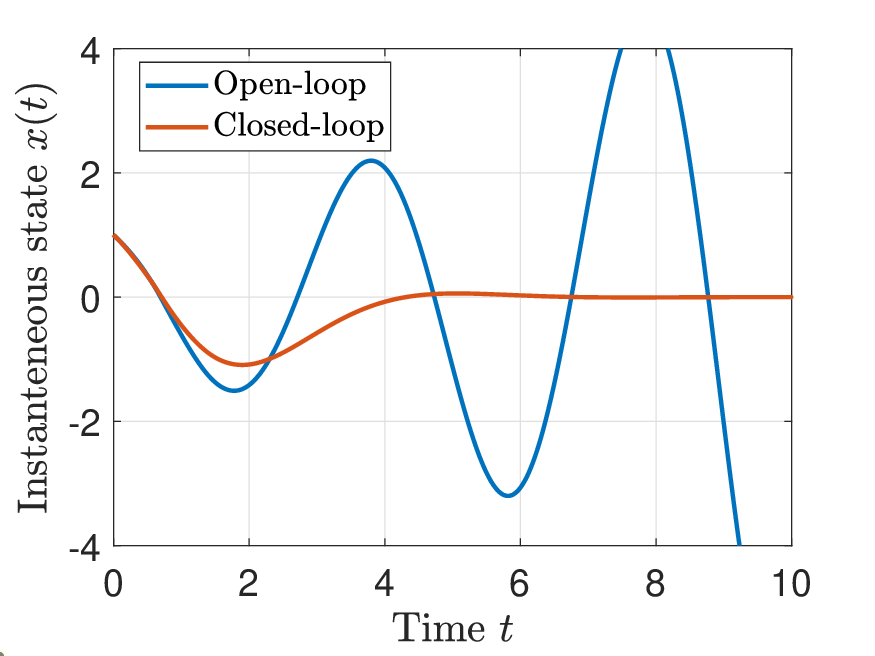}
    \caption{Solution $x(t)$ for ODE-transport case.}
    \label{fig3}
    \end{subfigure}
    \begin{subfigure}{0.45\linewidth}
    \centering
    \includegraphics[width=6cm]{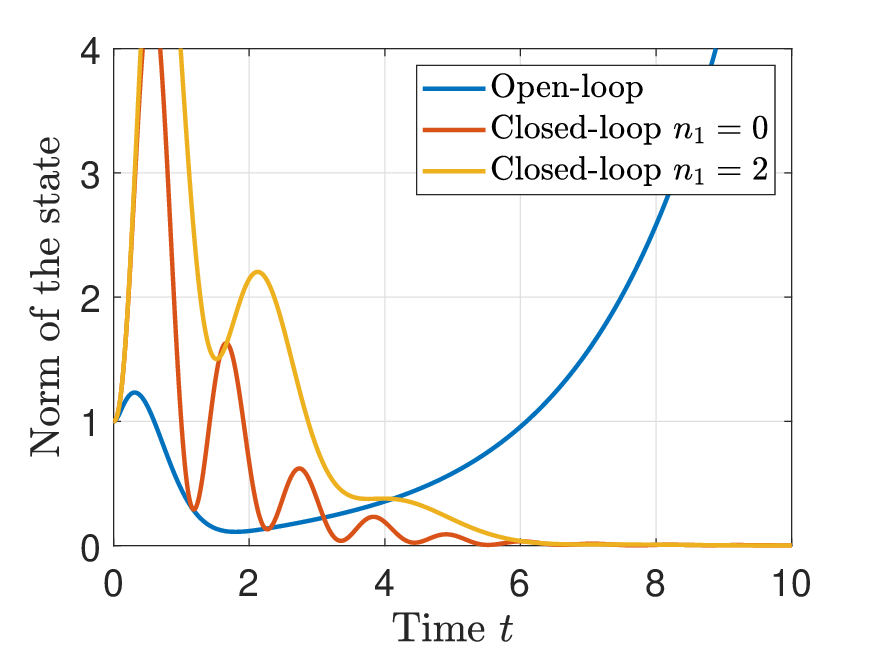}
    \caption{Solution $x(t)$ for ODE-reaction-diffusion case.}
    \label{fig4}
    \end{subfigure}
    \caption{ODE-PDE interconnected systems}
\end{figure}

\subsection{ODE-reaction-diffusion interconnection}

Consider system~\eqref{eq:system2} with for the ODE part
\begin{equation*}
    A = \begin{bsmallmatrix}0 & 1\\ -4 & -4\end{bsmallmatrix},\quad B= \begin{bsmallmatrix}0\\3\end{bsmallmatrix},\quad C = \begin{bsmallmatrix}1&0\end{bsmallmatrix},\quad
    B_u = \begin{bsmallmatrix}0\\1\end{bsmallmatrix},\quad C_y = \begin{bsmallmatrix}1&0\end{bsmallmatrix},
\end{equation*}
and for the PDE part $\lambda=\nu=1$.

As for the previous example, we design our controller in the light of a Padé approximated model. The PDE part, represented by the irrational transfer function $H(s) = \frac{\sqrt{\frac{s-\lambda}{\nu}}}{\sinh(\sqrt{\frac{s-\lambda}{\nu}})}$, is approximated by a rational transfer function at order $N=10$. 
This brings us to  the estimated matrices $(\hat{A}_0,\hat{A}_1,\hat{B}_0,\hat{B}_1,\hat{C}_0,\hat{C}_1)$ in the Jordan form.\\
Let $\delta = -1$. We have
\begin{equation}
\begin{array}{llll}
    &\hat{A}_0 = 0.2483, \quad &\hat{B}_0 = 0.0233, \quad &\hat{C}_0 = 1.9172,\\
    &\hat{A}_1 = \begin{bsmallmatrix} -1.5811 - 1.5285i & 0 \\ 0 & -1.5811 + 1.5285i \end{bsmallmatrix}, \quad &\hat{B}_1 = \begin{bsmallmatrix}0.0023 - 0.0345i\\0.0023 + 0.0345i\end{bsmallmatrix}, \quad &\hat{C}_1 = \begin{bsmallmatrix}-5.5760 + 1.2463i & -5.5760 - 1.2463i \end{bsmallmatrix},\\
\end{array}
\end{equation}
where $n_0=1$ and $n_1=2$. The model error is smaller than the Matlab precision.\\

In Fig.~\ref{fig4}, the stabilization effect of the finite-dimensional controller based on Padé approximation is illustrated.
From the initial condition $\begin{bsmallmatrix} x(0) \\ z(0,\theta)\end{bsmallmatrix} = \begin{bsmallmatrix}\begin{bsmallmatrix}1\\0\end{bsmallmatrix}\\0\end{bsmallmatrix}$, the blue plot represents the $\mathcal{H}=\mathbb{R}^2\times L^2(0,1)$ norm of the solution of the open-loop system with respect to the time. It diverges exponentially fast with a divergence rate given by $\hat{A}_0\simeq \frac{1}{4}$ (see~\cite[Section~IV]{Bajodek23}).
Designing the controller as described in Section~4 and selecting the gains~$(K_0,G_0)$ such that $$\hat{A}_0+\hat{B}_0K_0=-1,\qquad \hat{A}_0+G_0\hat{C}_0=-1,$$
we corroborate Corollary~\ref{cor2} obtaining a stable closed-loop system.
The red plot represents the $\mathcal{H}$ norm of the solution of this closed-looop system with $n_1=0$ or $n_1=2$ resulting in an exponential convergence rate.

\vspace{-0.2cm}
\section{Conclusions}
\vspace{-0.2cm}

This article has explored the design of finite-dimensional dynamical controllers using partial pole placement techniques for infinite-dimensional dynamical systems. Since the reduced model can accurately capture the essential dynamics of the original system and under some assumptions, we have shown that it suffices to selectively place the unstable poles.
The novelty of our approach comes from the use of a Lyapunov functional which strikes a balance between reduced-order modeling and control effectiveness. When properly designed and implemented,  the robustness of the closed-loop system with respect to model uncertainties has also been assessed. In conclusion, we have proposed a synthesis method of reduced-order controllers for high-dimensional plants.
\vspace{-0.2cm}

\bibliography{elsart-bib.bib}

\end{document}